\documentclass[12pt]{amsart}
\usepackage{amsmath,amssymb,
amsbsy,amsfonts,latexsym,amsopn,
amstext,cite,amsxtra,euscript,amscd,bm,mathabx}
\usepackage{url}

\usepackage[colorlinks,linkcolor=blue,
anchorcolor=blue,citecolor=blue,backref=page]{hyperref}
\usepackage{color}
\usepackage{graphics,epsfig}
\usepackage{graphicx}
\usepackage{float}
\usepackage{epstopdf}
\hypersetup{breaklinks=true}

\usepackage[norefs,nocites]{refcheck}

\newtheorem{theorem}{Theorem}[section]
\newtheorem{lemma}[theorem]{Lemma}

\newcommand{\R}{\mathbb{R}}

\newcommand{\Z}{\mathbb{Z}}

\newcommand{\wA}{\widehat{A}}

\newcommand{\E}{\mathsf{E}}
\newcommand{\D}{(\delta, \sigma)}

\title{Discretized sum-product for large sets}
\author[C. Chen]{Changhao Chen}

\address{Department of Pure Mathematics, University of New South Wales,
Sydney, NSW 2052, Australia}
\email{changhao.chenm@gmail.com}


\begin{document}

\begin{abstract} 
Let $A\subset [1, 2]$ be a $\D$-set with measure $|A|=\delta^{1-\sigma}$ in the sense of Katz and Tao. For $\sigma\in (1/2, 1)$ we show that 
$$
|A+A|+|AA|\gtrapprox \delta^{-c}|A|, 
$$
for $c=\frac{(1-\sigma)(2\sigma-1)}{6\sigma+4}$.  This improves  the bound of  Guth, Katz, and Zahl for large $\sigma$. 
\end{abstract}

\maketitle

\section{Introductions}

Erd\H{o}s-Volkmann \cite{EV}  showed that for any $\sigma\in [0, 1]$ there exists a  subgroup of reals  with Hausdorff dimension $\sigma$, and they conjectured that this property does not hold for subring of reals.  Precisely the Erd\H{o}s-Volkmann ring conjecture claims that there does not exists a subring of reals  with Hausdorff dimension strictly between zero and one.  Edgar and Miller \cite{EM}  first proved this ring conjecture via the orthogonal projections of fractal sets. A slightly later Bourgain \cite{Bourgain2003} independently proved the Erd\H{o}s-Volkmann ring conjecture via the discretized ring conjecture (discretized sum-product) of Katz and Tao \cite{KT}. 

Discretized sum-product also has many other applications.  For instance it is closely related to Falconer distance sets problem and the dimension of  Furstenburg sets, see Katz and Tao \cite{KT} for more details.   Bourgain \cite{Bourgain2010} showed a different approach for the discretized sum-product, and given applications in the projections of fractal sets and Fourier analysis. For the applications  of the discretized sets to projections of fractal sets see  He \cite{He}, Orponen \cite{Orponen}  and reference therein. For the applications of discretized sum-product to the Fourier decay of measures see Li \cite{Li}.

We note that  Bourgain \cite{Bourgain2003, Bourgain2010}  does not produce an explicit bound. Recently Guth, Katz, and Zahl \cite{GKZ} given a short proof of the discretized sum-product theorem, and they showed  an explicit bound for the discretized sum-product theorem. They \cite{GKZ}  used some  ideas from  Garaev \cite{G}  and  applied some discretized version of arguments from additive combinatorics.   

We show some notation of Katz and Tao \cite{KT} first. Let $\varepsilon, \delta$ be small and positive parameters.
We use $f\lessapprox g$ to denote $|f|\le C_\epsilon \delta^{-C \epsilon} |g|$, $f\gtrapprox g$ if $g\lessapprox f$, and 
$f \approx g$ if $f\lessapprox g$ and $g \gtrapprox f$. We say that a subset $A\subset \R$ is called $\delta$-discretized if $A$ is the union of intervals of lengths $\approx \delta$. For a positive constant $M$  we say that a subset  $A\subset (-M, M)$  is a $\D$-set if it is $\delta$-discretized, and one has 
$$
|A \cap B(x, r)|\lessapprox \delta(r/\delta)^{\sigma} 
$$
for all $\delta \le r<1$ and $x\in \R$. 

We note that a $\D$-set here may have much smaller measure than $\delta^{1-\delta}$. However we ask that any $\D$-set has positive measure to avoid to be an empty set. 

We remark that the $(\delta, \sigma)$-set can be considered as the discrete approximation of  set in $\R$ at scale $\delta$. Suppose that  $F$ is a compact subset of $\R$.  For each $k\in \Z$ let 
$$
I_{k, \delta}=[k\delta, (k+1)\delta],
$$
and $F_\delta$ be the union of the interval $I_{k, \delta}$ which intersects $F$, that is
\begin{equation*}
F_\delta=\bigcup_{ k\in \Z, 
F \cap I_{k, \delta}\neq \emptyset } I_{k, \delta}.
\end{equation*}
In many cases,  the set $F_\delta$ is a $(\delta, \sigma)$-set for some $0\le \sigma\le 1$. Here the parameter $\sigma$ is often related to the ``fractal dimension" of $F$.  For instance if $F$ is the classical Cantor ternary set, then for any $0<\delta<1$ the set $F_\delta$ is a $(\delta, \log2/ \log3)$-set. Indeed this follows from  Falconer \cite[Chapter 3]{Falconer}, which shows that the box dimension of the Cantor  set is $\log 2/\log 3$. We remark that  there are various dimensions in fractal geometry, we refer to \cite{Falconer} for more details. 

The above argument shows that the  $(\delta, \sigma)$-set in $\R^d$ is also important for understanding the structure of set in $\R^d$. However in this project we consider  $(\delta, \sigma)$-set on the line $\R$ only.

Let $A, B \subseteq \R$. The sum sets of $A$ and $B$ is defined as 
$$
A+B=\{a+b: a \in A, b\in B\},
$$
and similarly the product sets of $A$ and $B$ is defined as 
$$
AB=\{ab: a\in A, b \in B\}.
$$

Using the above notation, Bourgain's discretized sum product theorem claims that  
for a $\D$-set $A\subset [1, 2]$ with $0<\sigma<1$ and the measure $|A|=\delta^{1-\sigma}$, there exists a constant $c=c(\sigma)>0$ such that 
$$
|A+A|+|AA|\gtrapprox \delta^{-c} |A|.
$$
Under the same condition, Guth,  Katz, and  Zahl \cite{GKZ}  proved that  for any $c<\frac{\sigma(1-\sigma)}{4(7+3\sigma)}$ one has 
$$
|A+A|+|AA|\gtrapprox \delta^{-c} |A|.
$$
By adapting the arguments of Garaev \cite{G}, Guth,  Katz, and  Zahl \cite{GKZ}, and the bilinear bound of Bourgain \cite[Theorem 7]{Bourgain2010}, we obtain the following. 

\begin{theorem}
\label{thm:main}
Let $\sigma\in (1/2, 1)$ and $A\subset [1, 2]$ be a $\D$-set with measure $|A|=\delta^{1-\sigma}$. Then
$$
|A+A|+|AA|\gtrapprox \delta^{-\frac{(1-\sigma)(2\sigma-1)}{6\sigma+4}} |A|.
$$ 
\end{theorem}

Note that for any $1/2<\sigma<1$, Theorem \ref{thm:main} gives a non-trivial lower bound. Furthermore  for  $\sigma>\frac{\sqrt{226}-10}{9}=0.5703\ldots$, Theorem \ref{thm:main} improves the bound of Guth,  Katz, and  Zahl  \cite{GKZ}.

\section{Preliminaries}

We use $\#S$ to denote the cardinality of set $S$.  Let $X, Y, Z\subset \R$ be finite sets with $Z\neq \emptyset$, 
then the Ruzsa triangle inequality claims that 
$$
\#(X+Y)\le \frac{\#(X+Z) \#(Z+Y)}{\#Z}.
$$
See \cite[Chapter 2]{TV} for a proof and  many other useful sum sets estimates. 

We need the following well known discretized version of Ruzsa triangle inequality. Our proof is based on 
Guth,  Katz, and  Zahl \cite[Proof of Corollary 2.3]{GKZ},  Orponen \cite[Remark 4.40]{Orponen}. 
For many other  discretized version of sum sets estimates see He \cite{Hethesis},  Tao \cite{Taoblog}. 

We  show a geometric observation first. 
Let $S\subset \R$ be the union of  disjoint intervals with length $\gtrapprox \delta$.  Then for all $0<c<10$ we have
\begin{equation}
\label{eq:nnn}
|S+B(0, c\delta)|\lessapprox |S|.
\end{equation}
Here  we can change the parameter $10$ to any other fixed positive constant.

\begin{lemma}
\label{lem:Ruzsa}
Let $A, B, C \subset \R$ be  $\delta$-discretized sets.  Then 
$$
|A+B|\lessapprox \frac{|A+C||C+B|}{|C|}.
$$
\end{lemma}
\begin{proof}
With out losing general we may assume that each interval of $A, B$ and $C$ has length at least $\delta$. In the end we change the estimate $\ll$ to $\lessapprox$, and this does not change our result.

For any set $S\subset \R$ let $S_{\delta}=(\delta/3)\mathbb{Z} \cap S$.  For any $a\in A, b\in B$ there exists $a'\in A_\delta, b'\in B_\delta$ such that 
$$
|a+b-a'-b'|\le \delta.
$$
It follows that 
\begin{equation}
\label{eq:Adelta}
A+B\subset A_\delta+B_\delta+B(0, \delta)\subset A+B+B(0, 2\delta).
\end{equation}
Combining with \eqref{eq:nnn} we obtain
$$
|A+B|\ll \#(A_\delta+B_\delta) \delta\ll |A+B|.
$$
Applying Ruzsa triangle inequality to sets $A_\delta, B_\delta, C_\delta$ and applying \eqref{eq:Adelta} to $A+C$ and $C+B$ we obtain the result.
\end{proof}

\begin{lemma}
\label{lem:continuous}
Let $A\subset [0, 2]$ be a $\delta$-discretized set, and  the union of the intervals of $A$ are pairwise disjoint. Then for any $t\in [0.5, 2.1]$ and any $x\in [-\delta, \delta]$ we have 
$$
|A+tA|\lessapprox |A+(t+x)A|\lessapprox |A+tA|.
$$
\end{lemma}
\begin{proof}
For any $t\in [0.5, 2], x\in [-\delta, \delta]$  and $a, b \in A$ we have
$$
|(a+tb) - (a+(t+x)b)|\le 2\delta,
$$
and hence 
$$
A+tA \subset A+(t+x)A +B(0, 2\delta)\subset A+tA+B(0, 4\delta).
$$
Combining with \eqref{eq:nnn} we finish the proof.
\end{proof}

 Let $f: \R\rightarrow \R$ be a function. The Fourier transform of the function $f$ at $\xi \in \R$ is defined as 
$$
\widehat{f}(\xi)=\int_{\R} e(-x\xi) f(x)dx,
$$
where throughout the paper we
denote $e(x)=e^{2\pi i x}$. Let $\mu$ be a measure on $\R$. The Fourier transform of the measure $\mu$ at  $\xi\in \R$ is defined as 
$$
\widehat{\mu}(\xi)=\int e(-x\xi)d \mu(x).
$$

For a subset $S\subset \R$ we will also use $S$ to denote the characteristic  function of $S$. 
Let $A, B\subset \R$ be two bounded sets. The convolution of $A$ and $B$ is defined as 
$$(A\ast B)(x)=\int_{\R} A(x-y)B(y)dy.$$ 
The (additive) energy of $A, B$ is defined as 
\begin{equation}
\label{eq:energy}
 E(A, B)=\int_{\R} (A\ast B)(x)^{2}dx= \int_{\R} |\widehat{A}(\xi)|^{2} |\widehat{B}(\xi)|^{2} d \xi.
\end{equation} 
The second equality holds by applying the Plancherel identity and convolution theorem. Clearly we have 
$$
\int_{\R}(A\ast B)(x) dx=|A||B|.
$$
By Cauchy-Schwarz inequality we obtain
\begin{equation}
\label{eq:energyandsumset}
(|A||B|)^{2}\le E(A, B)|A+B|.
\end{equation}

We will frequently use the  Plancherel identity for a set.  Precisely for a bounded subset $S\subset \R$ we have 
\begin{equation}
\label{eq:Plan}
\int_{\R} |\widehat{S}(\xi)|^{2}d \xi= |S|.
\end{equation}

We formulate the following version  of  Bourgain \cite[Theorem 7]{Bourgain2010}. Note that the interval $[0,4]$ is not essential, in fact Lemma~\ref{lem:anyinterval} holds for any bounded interval.

\begin{lemma}
\label{lem:anyinterval}
Let $\mu, \nu$ be probability measures on $[0, 4]$ such that for all $\delta <r\le 1 $ and all $x\in \R$,
\begin{equation*}
\mu(B(x, r))\lessapprox K_1 r^{\alpha}\text{ and } \nu(B(x, r))\lessapprox K_2 r^{\beta}.
\end{equation*}
Then for $1\le |\xi|\le \delta^{-1}$ we have 
$$
 \int_{\R} \left| \int_{\R} e(-xy\xi) d\mu(x)\right| d\nu(y)  \lessapprox \sqrt{K_1K_2} ~ |\xi|^{-\frac{\alpha+\beta-1}{2}}.
$$
\end{lemma}

We remark that the statement of \cite[Theorem 7]{Bourgain2010} gives a bound for 
$$
\left| \int_{\R}  \int_{\R} e(-xy\xi) d\mu(x)d\nu(y)\right|.
$$
However the proof of \cite[Theorem 7]{Bourgain2010} (see \cite[(8.3)]{Bourgain2010}) indeed works for the bound
$$
\int_{\R} \left| \int_{\R} e(-xy\xi) d\mu(x)\right| d\nu(y),
$$
which is used for bounding~\eqref{eq:Bourgain} below. Moreover the term $\sqrt{K_1K_2}$ appears,  since the using of  Cauchy-Schwarz inequality in the proof of \cite[Theorem 7]{Bourgain2010}.

In particularly we have the following version for $\D$-sets which is easier for our  using.

\begin{lemma}
\label{lem:bilinear}
Let $0<\alpha, \beta<1$. Let $A\subset [0,4]$ and $B\subset [0,4]$ be $(\delta, \alpha)$-set and $(\delta, \beta)$-set respectively. Then for any $1\le \xi\le \delta^{-1}$ we have 
$$
 \int_{A}\left| \int_{B} e(-xy\xi) dx \right| dy  \lessapprox  |\xi|^{-\frac{\alpha+\beta-1}{2}}\delta^{\frac{2-\alpha-\beta}{2}}\sqrt{|A||B|}.
$$
\end{lemma}
\begin{proof}
For $A$ we define a measure $\mu$ by letting
$$
\mu (X)=\frac{|X\cap A|}{|A|} \text{ for } X\subset \R.
$$
By the condition that $A$ is a $(\delta, \alpha)$-set we obtain 
\begin{equation}
\label{eq:mu}
\mu(B(x, r)) \lessapprox r^{\alpha}\delta^{1-\alpha}|A|^{-1} \text{ for } \delta\le  r\le 1.
\end{equation}

Similarly for $B$ we define a measure $\nu$ by letting 
$$
\nu (X)=\frac{|X\cap A|}{|A|} \text{ for } X\subset \R,
$$
and we have 
\begin{equation}
\label{eq:nu}
\nu(B(x, r)) \lessapprox r^{\beta}\delta^{1-\beta}|B|^{-1} \text{ for } \delta\le  r\le 1.
\end{equation}
Clearly we have 
$$
 \int_{A} \left|\int_{B} e(-xy\xi) dx \right| dy  = |A||B| \int_{A} \left| \int_{B} e(-xy\xi) d\mu(x) \right| d \nu(y).
$$
Then  Lemma \ref{lem:bilinear} and estimates \eqref{eq:mu}, \eqref{eq:nu} give the result.
\end{proof}

\section{Proof of Theorem \ref{thm:main}}

By adapting the arguments in Guth, Katz, and Zahl \cite{GKZ} and especially in Garaev \cite{G}  we obtain the following.

\begin{lemma}
\label{lem:G}
Let $0<\sigma<1$. Let $A\subset [1, 2]$ be a $\D$-set with measure $|A|=\delta^{1-\sigma}$. Then there exists a $\D$-set $T\subset [0.4, 2.1]$ with measure 
$$
|T|\gtrapprox \frac{|A|^{2}}{|AA|}
$$ such that for any $t\in T$ we have 
$$
|A+tA|\lessapprox \frac{|A+A|^{2}|AA|}{|A|^{2}}.
$$
\end{lemma}
\begin{proof}
Let $D=\{a_1, \ldots, a_N\}$ be a maximal $\delta$-separated subset of $A$, i.e., any two distinct elements of $D$ has distance at least $\delta$, furthermore for any $a\in A$ there exists $a_i\in D$ such that $|a-a_i|\le\delta$. Note that  
\begin{equation}
\label{eq:N}
N \approx \delta^{-\sigma},
\end{equation}
and for all $\delta \le r\le 1$ and  $x\in \R$, 
\begin{equation}
\label{eq:D}
\# \left(D \cap B(x, r) \right) \lessapprox \left(\frac{r}{\delta} \right)^{\sigma}.
\end{equation}

Let 
$$
f(x)=\sum_{i=1}^{N} {\bf 1}_{a_i A}(x).
$$
Then we have 
\begin{equation}
\label{eq:first}
\int_{\R} f(x)dx=\sum_{i=1}^N|a_i A|\ge N|A|,
\end{equation}
and 
\begin{equation}
\label{eq:second}
\int_{\R} f(x)^2dx= \sum_{1\le i, j\le N} |a_iA\cap a_j A|.
\end{equation}
By Cauchy-Schwarz inequality  we arrive 
\begin{equation}
\label{eq:CSnew}
\left(\int_{\R} f(x)dx \right)^2\le \int_\R f(x)^2 dx \left|\{x\in \R: f(x)>0\} \right|.
\end{equation}
Observe that 
$$
\{x\in \R: f(x)>0\} \subset AA.
$$
Thus together with \eqref{eq:first}, \eqref{eq:second}, and \eqref{eq:CSnew}, we obtain 
$$
\sum_{1\le i, j\le N} |a_i A \cap a_j A|\ge \frac{N^2|A|^{2}}{|AA|}.
$$
Thus there exists $1\le j_0\le N$ such that 
$$
\sum_{1\le i\le N} |a_i A \cap a_{j_0} A|\ge \frac{N|A|^{2}}{|AA|}.
$$
Let 
$$
P=\{1\leq i\le N: |a_i A\cap a_{j_0}A|\ge \frac{|A|^{2}}{2|AA|}\}.
$$
Then 
$$
\sum_{i\in P} |a_i A \cap a_{j_0} A|\ge \frac{N|A|^{2}}{2|AA|}.
$$
Taking dyadic decomposition for  $|a_i A \cap a_{j_0} A|$ with $i\in P$,  we obtain 
$$
\sum^{K}_{k=1}2^{-k}\# P_k\gg \sum_{i\in P} |a_i A \cap a_{j_0} A|\ge\frac{N|A|^{2}}{2|AA|}, 
$$
where 
$$
P_k=\{i \in P:  2^{-k-1}<|a_i A \cap a_{j_0} A|\le 2^{-k}\},
$$
and $K$ is an integer parameter such that 
$$
2^{-K-1}<  \frac{|A|^{2}}{2|AA|}\le 2^{-K}.
$$
Since the product set $AA$ is a subset of $ [1, 4]$,  we have $|AA|\ll 1$. 
It follows that 
$$
K\ll \log\left(\frac{1}{\delta}\right).
$$
Thus there exist $\tau$ and $D_\tau\subset P$ such that 
\begin{equation*}
\tau \# D_\tau \gtrapprox \sum^{K}_{k=1}2^{-k}\# P_k\gg \frac{N|A|^{2}}{|AA|},
\end{equation*}
and for any $i\in D_\tau$ we have 
\begin{equation}
\label{eq:tau}
\tau\le |a_i A \cap a_{j_0} A|\le 2 \tau.
\end{equation}
Since $\tau \le|A|$ and $\#D_\tau \le N$, we obtain 
\begin{equation}
\label{eq:taug}
\tau \gtrapprox \frac{|A|^{2}}{|AA|},
\end{equation}
and 
\begin{equation}
\label{eq:Dtau}
\# D_\tau \gtrapprox \frac{N|A|}{|AA|}.
\end{equation}

Now we intend to bound the measure of the set $a_iA+a_{j_0}A$ for each $i\in D_{\tau}$. For this purpose we introduce some notation first.

For each $k\in \Z$ let 
$
J_{k, \delta}=[k\delta, (k+1)\delta).
$
For each $1\le i\le N$ let 
$$
U_i=\bigcup_{k\in \Z, \,a_i A \cap J_{k,\delta}\neq \emptyset } J_{k, \delta}.
$$
Note that   
\begin{equation}
\label{eq:AU}
a_iA \subseteq U_i \subseteq a_iA +B(0, \delta),
\end{equation}
and the intersection $U_i\cap U_j$ is a $\delta$-discretized  set for  $1\le i, j\le N$. Moreover 
 for each  $i\in D_{\tau}$  by \eqref{eq:tau}, \eqref{eq:taug} we have
\begin{equation}
\label{eq:non-zero}
|U_i\cap U_{j_0}| \ge |a_iA \cap a_{j_0} A| \gtrapprox \frac{|A|^2}{|AA|}.
\end{equation}

For any $i\in D_\tau$ applying  Ruzsa triangle inequality    Lemma \ref{lem:Ruzsa} for the sets $a_iA, a_{j_0} A$ and $U_i\cap U_{j_0}$, we derive 
\begin{equation}
\label{eq:sumset}
|a_iA+a_{j_0}A | \lessapprox \frac{|a_i A+ U_i\cap U_{j_0} | |U_i\cap U_{j_0} +a_{j_0}A|}{|U_i\cap U_{j_0} |}.
\end{equation}
By \eqref{eq:AU} we have 
$$
a_iA+ U_i\cap U_{j_0} \subseteq a_iA+a_iA+B(0, \delta).
$$
Since $1\le a_i\le 2$ and the simper fact \eqref{eq:nnn} we obtain 
$$
|a_iA+ U_i\cap U_{j_0}|\ll |A+A|.
$$
Similarly, we have 
$$
|a_{j_0}A+U_i\cap U_{j_0} |\ll |A+A|.
$$
Combining with \eqref{eq:non-zero} and \eqref{eq:sumset} we arrive
\begin{equation*}
\begin{split}
|a_iA+a_{j_0}A|\lessapprox \frac{|A+A|^2|AA|}{|A|^{2}}.
\end{split}
\end{equation*}

Applying Lemma \ref{lem:continuous} we obtain that for any $i\in D_\tau$ and $x \in (-\delta, \delta)$ 
we have 
$$
|A+(a_i/a_{j_0}+x)A| \lessapprox \frac{|A+A|^2|AA|}{|A|^{2}}.
$$
Let  
$$
T=\bigcup_{i\in D_\tau} B(a_i/a_{j_0}, \delta/2).
$$
We ask that $\delta$ is a small positive parameter, and hence $T\subset [0.4, 2.1]$.
Furthermore, the estimates \eqref{eq:N},  \eqref{eq:Dtau} imply
$$
|T| \gg \# D_\tau \delta \gtrapprox \frac{|A|^{2}}{|AA|}.
$$
By \eqref{eq:D} we obtain that $T$ is a $\D$-set which finishes the proof.
\end{proof}

Applying Lemma \ref{lem:bilinear} we obtain the following  upper bound of the  mean value of energies $E(A, tA)$.

\begin{lemma} 
\label{lem:B} Fix $1/2<\sigma <1$.
Let $T\subset [0.4, 2.1]$ be a $\D$-set with the measure $|T|\ge \delta$. Let $A\subset [1, 2]$ be a $\D$-set with measure $|A|=\delta^{1-\sigma}$. Then 
$$
\int_{T} \mathsf{E} (A, tA) dt \lessapprox |A|^{3}|T|  \left (|A|^{\frac{2\sigma}{1+2\sigma}}|T|^{-\frac{1}{1+2\sigma}}+\delta(|A||T|)^{-1} \right).
$$
\end{lemma}
\begin{proof}
For each $t\in \R$ and $\xi\in \R$ we have 
$$
\widehat{tA}(\xi)=t\widehat{A}(t\xi).
$$
Thus by \eqref{eq:energy} and the condition $T\subset [0.4, 2.1]$ we conclude that 
\begin{equation}\label{eq:AE}
\begin{split}
\int_T\E(A, tA)dt \ll \int_{T}\int_{\R} |\wA(\xi)|^{2}|\wA(t\xi)|^{2}dt d\xi.
\end{split}
\end{equation}

Let $0<L<1/\delta$ be a parameter which will be determined later. We decompose $\R$ into three parts, and then bound  \eqref{eq:AE} by  three  corresponding  parts. Precisely, 
\begin{equation} 
\label{eq:three}
\eqref{eq:AE}\ll I_0+I_1+I_2
\end{equation}
where 
\begin{equation*}
I_0=\int_{T}\int_{|\xi|\le L} |\wA(\xi)|^{2}|\wA(t\xi)|^{2}dt d\xi,
\end{equation*}
\begin{equation*}
I_1=\int_{T}\int_{L\le |\xi|\le \delta^{-1}} |\wA(\xi)|^{2}|\wA(t\xi)|^{2}dt d\xi,
\end{equation*}
and 
\begin{equation*}
I_2=\int_{T}\int_{|\xi|\ge \delta^{-1}} |\wA(\xi)|^{2}|\wA(t\xi)|^{2}dt d\xi.
\end{equation*}

For $I_0$, we use the trivial bound $|\widehat{A}(\xi)|\le |A|$, and we obtain  
\begin{equation}
\label{eq:I0}
I_0\le |A|^{4}|T|L.
\end{equation}

For $I_1$, clearly the trivial bound $|\widehat{A}(\xi)|\le |A|$ gives 
\begin{equation}
\label{eq:Bourgain}
\int_{T} |\widehat{A}(\xi t)|^{2} dt\le  |A|\int_{T} |\widehat{A}(\xi t)| dt.
\end{equation}
Applying Lemma \ref{lem:bilinear}, and the condition $|A|=\delta^{1-\sigma}$, we obtain 
$$
\int_{T} |\widehat{A}(\xi t)| dt\lessapprox |A|^{3/2}|T|^{1/2}|\xi|^{-(2\sigma-1)/2}.
$$
Thus we arrive
$$
\eqref{eq:Bourgain} \lessapprox |A|^{5/2}|T|^{1/2}|\xi|^{-(2\sigma-1)/2}.
$$
Combining with Fubini's theorem and the condition $\sigma>1/2$, we have
\begin{equation*}
\begin{split}
I_1&\lessapprox |A|^{5/2}|T|^{1/2} \int_{L\le |\xi|\le \delta^{-1}} |\widehat{A}(\xi)|^{2} |\xi|^{-(2\sigma-1)/2}d \xi \\
&\lessapprox |A|^{5/2}|T|^{1/2} L^{-(2\sigma-1)/2} \int_{L\le |\xi|\le \delta^{-1}} |\widehat{A}(\xi)|^{2} d\xi.
\end{split}
\end{equation*}
Plancherel identity~\eqref{eq:Plan} implies 
$$
\int_{L\le |\xi|\le \delta^{-1}} |\widehat{A}(\xi)|^{2} d\xi \le |A|.
$$
Thus we arrive 
\begin{equation}
\label{eq:I1}
I_1\lessapprox |A|^{7/2}|T|^{1/2} L^{-(2\sigma-1)/2}.
\end{equation}

Now we optimize the choice of the parameter $L$ to find the smallest upper bound for the parts  $I_0, I_1$. Recalling that we  ask  $0<L<1/\delta$.  In the end, the parameter $L_{0}$, which  makes the right hand sides of~\eqref{eq:I0},~\eqref{eq:I1} ``comparable", satisfies our need. Thus  we derive 
\begin{equation}
\label{eq:L}
L_0=(|A||T|)^{-1/(1+2\sigma)}.
\end{equation}
Indeed the conditions $|A|=\delta^{1-\sigma}$, $|T|\ge \delta$, and $1/2<\sigma<1$, imply that $L_0\le \delta^{-1}$. It follows that 
\begin{equation}
\label{eq:i0i1}
I_0, I_1\lessapprox |A|^4|T|L_0\lessapprox |A|^{3}|T|  (|A|^{\frac{2\sigma}{1+2\sigma}}|T|^{-\frac{1}{1+2\sigma}}).
\end{equation}

Now we turn to the estimate for $I_2$. By changing variables and applying the Plancherel identity we obtain
$$
\int_{\R} |\widehat{A}(t\xi)|^{2}dt \le |A||\xi|^{-1}.
$$
Again by applying the Plancherel identity  we have 
\begin{equation*}
\begin{split}
I_2 \le \int_{|\xi|\ge \delta^{-1}} |\widehat{A}(\xi)|^{2} |A| |\xi|^{-1}d \xi \ll |A|^{2}\delta.
\end{split}
\end{equation*}
Combining with \eqref{eq:three}, \eqref{eq:I0}, \eqref{eq:L}, \eqref{eq:i0i1}, we obtain the desired bound.
\end{proof}

Now we turn to the proof of Theorem \ref{thm:main}. Suppose that 
\begin{equation*}
\max\{|A+A|, |AA|\}=K |A|.
\end{equation*}
By Lemma \ref{lem:G} there exists a $(\delta, \sigma)$-set $T\subset [0.4, 2]$ such that 
\begin{equation}
\label{eq:T}
|T|\gtrapprox |A|/K,
\end{equation}
and  for each $t\in T$ we have 
\begin{equation}
\label{eq:AtA}
|A+tA|\lessapprox K^{3}|A|.
\end{equation}
Applying  Lemma \ref{lem:B} to $A$ and $T$, we conclude  that  there exists a $t_0\in T$ such that 
$$
E(A, t_0A)\lesssim |A|^{3}  \left (|A|^{\frac{2\sigma}{1+2\sigma}}|T|^{-\frac{1}{1+2\sigma}}+\delta(|A||T|)^{-1} \right).
$$
By \eqref{eq:energyandsumset} and estimates \eqref{eq:T}, \eqref{eq:AtA}, we obtain that 
\begin{equation*}
\begin{split}
|A|&\lessapprox \left (|A|^{\frac{2\sigma}{1+2\sigma}}|T|^{-\frac{1}{1+2\sigma}}+\delta(|A||T|)^{-1} \right)|A+t_0A|\\
& \lessapprox |A|^{\frac{2\sigma-1}{2\sigma+1}}|A|K^{\frac{6\sigma+4}{2 \sigma+1}}+|A|^{-1}\delta K^{4}.
\end{split}
\end{equation*}
It follows that 
$$
K\gtrapprox \min\{ \delta^{-\frac{(1-\sigma)(2\sigma-1)}{6\sigma+4}}, \delta^{-\frac{2\sigma-1}{4}}\}.
$$
Note that for $1/2<\sigma<1$ we have 
$$
\frac{(1-\sigma)(2\sigma-1)}{6\sigma+4} \le \frac{2\sigma-1}{4},
$$
and hence
$$
K\gtrapprox \delta^{-\frac{(1-\sigma)(2\sigma-1)}{6\sigma+4}},
$$
which gives the result.

\section*{Acknowledgement}

The author is  grateful to  Igor Shparlinski  for his
comments on the initial draft. Especially the author would like to thank the anonymous referee for carefully reading the manuscript and giving excellent comments, and thus improving the quality of this article.

This work was  supported in part  by ARC Grant~DP170100786.

\end{document}